\documentclass[11pt]{amsart}

 \newcommand{\private}[1]{}   

\usepackage{amssymb, amsmath, amscd, amsthm}
\usepackage{MnSymbol}
\usepackage[all]{xy}          
\xyoption{dvips}              

\newcommand{\calC}{{\mathcal{C}}}

\numberwithin{equation}{section}

\DeclareMathOperator{\hoTot}{hoTot}
\DeclareMathOperator{\hofibre}{hofibre}

\newcommand{\stirlingone}[2]{\genfrac{[}{]}{0pt}{}{#1}{#2}}
\newcommand{\stirlingtwo}[2]{\genfrac{\{}{\}}{0pt}{}{#1}{#2}}
\newcommand{\eulernb}[2]{\genfrac{\llangle}{\rrangle}{0pt}{}{#1}{#2}}

\newcommand{\BK}{{\mathbb K}}
\newcommand{\BR}{{\mathbb R}}
\newcommand{\BN}{{\mathbb N}}

\newcommand{\Embd}{\operatorname{Emb}_\partial}
\newcommand{\Immd}{\operatorname{Imm}_\partial}
\newcommand{\width}{\operatorname{width}}

\newcommand{\Ho}{\operatorname{H}}

\newcommand{\map}{{\operatorname{map}}}
\newcommand{\im}{{\operatorname{im}}}
\newcommand{\Conf}{{\operatorname{Conf}}}
\newcommand{\Confd}{{\operatorname{Conf}_{\partial}}}

\newcommand{\Tot}{\operatorname{Tot}}

\newcommand{\hookto}{\hookrightarrow}

\newcommand{\KK}{{\mathcal{K}}}
\newcommand{\LL}{{\mathcal{L}}}
\newcommand{\uu}[1]{\underline{{#1}}}
\newcommand{\uuu}[1]{\underline{\underline{{#1}}}}
\newcommand{\HBKSS}{$\Ho^*\operatorname{BKSS}$\ }

\theoremstyle{plain}
\newtheorem{thm}{Theorem}[section]

\newtheorem{prop}[thm]{Proposition}

\newtheorem{lemma}[thm]{Lemma}
\newtheorem{cor}[thm]{Corollary}

\theoremstyle{definition}

\theoremstyle{remark}

\newcommand{\refT}[1]{Theorem~\ref{T:#1}}
\newcommand{\refC}[1]{Corollary~\ref{C:#1}}
\newcommand{\refP}[1]{Proposition~\ref{P:#1}}

\newcommand{\refL}[1]{Lemma~\ref{L:#1}}

\newcommand{\refN}[1]{(\ref{#1})}

\begin{document}


\title{Euler series, Stirling numbers and the growth of the homology of the space of long links.}


\author{Guillaume Komawila}
\address{CIPMA}
\email{komawila@yahoo.fr} 
\author{Pascal Lambrechts}
\address{Institut Math\'{e}matique, 2 Chemin du Cyclotron, B-1348 Louvain-la-Neuve, Belgium}
\email{lambrechts@math.ucl.ac.be}
\urladdr{http://milnor.math.ucl.ac.be/plwiki}


\begin{abstract}
We study the Bousfield-Kan spectral sequence associated to the
Munson-Voli\'c cosimplicial model for the space of long links with
$\ell$ strings in $\BR^N$. We compute explicitely some
Euler-Poincar\'e series associated to the second page of that spectral
sequence and deduce exponential growth of their Betti numbers.
\end{abstract}

\maketitle



\section{Introduction}
A long link with $\ell$ strings in $\BR^N$ is a smooth embedding of
$\ell$ disjoint copies of $\BR$ in $\BR^N$ with a fixed behaviour at
infinity. This generalizes the classical notion of long knots. In
\cite{Sin:TSK} and \cite{Sin:OKS}, D.~Sinha has
constructed a cosimplicial model for the space of long knots, and
V.~Turchin, I.~Voli\'c and the second author have proved in 
\cite{LTV:HQLK}  that the associated cohomology Bousfield-Kan spectral
sequence
collapses at the $E_2$ page.
B.~Munson and I.~Voli\'c, \cite{MuVo:HoSt}, have constructed a cosimplicial model for the
space of long links with $\ell$ strings, in the same spirit as 
Sinha's cosimplicial model for the space of long knots.
In this paper, we study  the cohomology Bousfield-Kan spectral sequence
associated to the Munson-Voli\'c cosimplicial model. 

Our main result is an explicit closed
formula for the Euler characteristic  of each line of the second page
of that spectral sequence. This computation is based on some
interesting relation between the dimensions of the summands of the
$E_2$-page of the spectral sequence and some Stirling numbers.
Since these Euler characteristic have an
exponential growth, we deduce an exponential growth of the dimension
of the associated graded space of $E_2$. 
We plan to prove in a future
paper
that the spectral sequence collapses at the $E_2$ page, generalizing
the main result of  \cite{LTV:HQLK}. This will imply the exponential
growth of the Betti numbers of the space of long links, and even of
long links modulo $\ell$-fold product of the space of long knots
(which is a retract of the space of long links, as we will see in this
paper.)

\subsection{Acknowledgment}
The present work is supported in part by a ICMPA-UCL Project between the International
Chair in Mathematical Physics and Applications (ICMPA-UNESCO Chair) at the University of
Abomey-Calavi (Cotonou, Benin), and the Institute of Research in Mathematics and Physics
(IRMP) at the Catholic University of Louvain (UCL, Louvain-la-Neuve, Belgium), the latter
being the Institution financing entirely this Project whose purpose is to support the training
of mathematicians and physicists from the Democratic Republic of Congo, in particular at the
University of Kinshasa (Kinshasa, DRC). 
The work of Komawila is supported by a PhD Fellowship held alternatively at both Institutions
within the framework of this ICMPA-UCL Project. 
We thank Prof. Norbert Hounkounou and Jan Govaerts for making this
collaboration possible.
The results of this paper have been presented
 in a Fields institute conference  celebrating Yves Felix's birthday
 and  organized by
Barry Jessup and Paul-Eugene Parent.
We also thank M. Kauers for the proof of \refP{S1S2}. 

\section{Long knots and long strings}

Fix $N\geq3$. A long knot is  a smooth embedding
\[f\colon\BR\hookto\BR^N\]
satisfying the  boundary conditions
\begin{equation}
\label{E:BCK}
\begin{cases}
  f([-1,1])\,\subset\,[-1,1]\times\BR^{N-1}\\
f(t)=(t,0,\dots,0)&\textrm{if }|t|\geq1.
\end{cases}
\end{equation}
In other words $f$ should be thought of as a string parametrized by
$[-1,1]$ and knotted inside the box
$[-1,1]\times\BR^{N-1}$
with its two extremities attached to the points $(\pm1,0,\dots,0)$.
The space of all these embeddings, endowed with a suitable topology
(the weak $\calC^\infty$-topology), is denoted by
\begin{equation}
\label{E:EmbdR}
\Embd(\BR,\BR^N)
\end{equation}
where the decoration $\partial$ is a reminder of the boundary
conditions \refN{E:BCK}.

We generalize this to multiple strings. Let $\ell\geq1$ be an integer
and set $[\ell]:=\{0,\dots,\ell-1\}$ a set of cardinality $\ell$. A \emph{long link with $\ell$
strings} is a smooth embedding
\[g\colon\BR\times[\ell]\hookto\BR^N\]
satisfying the following boundary conditions
\begin{equation}
\label{E:BCL}
\begin{cases}
  g([-1,1]\times[\ell])\,\subset\,[-1,1]\times\BR^{N-1}\\
g(t,i)=(t,i,0,\dots,0)&\textrm{if }|t|\geq1\textrm{ and }i=0,\dots,\ell-1.
\end{cases}
\end{equation}
We denote the space of all these embeddings by
\[\Embd(\BR\times[\ell],\BR^N).\]
For $\ell=1$ this space coincide with the space $\refN{E:EmbdR}$.

\subsection{The $\ell$-fold product of the space of long knots as a
  retract up to homotopy of the space of long links}
Each of the $\ell$ strings of a link can be seen as a long knot. More
precisely, given a long link $g\in\Embd(\BR\times[\ell],\BR^N)$, we can take its
restriction to the $i$th string 
\[g_{(i)}:=\left(g|\BR\times\{i\}\right)\colon\BR\times\{i\}\hookto\BR^N\]
for $i=0,\dots,\ell-1$. For $s\in\BR$, consider the translation of
$\BR^N$ in
the second coordinate direction
\begin{equation}
\label{E:tau}
\tau_s\colon\BR^N\to\BR^N\,,\,(x_1,\dots,x_N)\mapsto(x_1,x_2+s,x_3,x_4,\dots,x_N).
\end{equation}
Then the composite $\tau_{-i}\circ g_{(i)}$ is a long knot (after
identifying $\BR\times\{i\}$ with $\BR$.)
Hence we get a map
\begin{eqnarray}
\label{E:rho}
\rho\colon\Embd(\BR\times[\ell],\BR^N)&\longrightarrow&
\Embd(\BR,\BR^N)^{\times\ell}\\
\notag
g&\mapsto&(\tau_0\circ g_{(0))},\dots,\tau_{-(\ell-1)}\circ g_{(\ell-1)}),
\end{eqnarray}
which associates to a long link the $\ell$-tuple of the restricted long
knots.

Actually $\rho$ is a retraction up to homotopy. Indeed we can define
a map in the other direction
\begin{equation}
\label{E:iota}
\iota\colon\Embd(\BR,\BR^N)^{\times\ell}
\to
\Embd(\BR\times[\ell],\BR^N)
\end{equation}
which sends an $\ell$-tuple of long knots to a link obtained by
juxtaposition of these knots, avoiding to link any two strings.
Let us define this map $\iota$ more precisely. Given a long knot $f$
define its  \emph{width} as the number
\[
\width(f):=\inf\{r\geq0:f([-1,1])\,\subset\,[-1,1]\times[-r,r]^{N-1}\},
\]
which depends
continuously on $f$.
For $R>0$ consider the linear map
\[H_R\colon\BR^N\to\BR^N\,,\,(x_1,x_2,\dots,x_N)\mapsto (x_1,R\cdot
x_2,\dots,R\cdot x_N)\]
which is the product of the identity map on $\BR$ and a homothety of
rate $R$ on $\BR^{N-1}$.
If $f$ a long knot then $H_R\circ f$ is also a long knot and
\[
\width(H_R\circ f)=R\cdot\width(f).
\]
In particular, for $R=\min(\width(f),1/4)$,
\[
H_{\min(\width(f),1/4)}\circ f
\]
is a long knot of width $\leq1/4$.

We now define the map
$\iota$ of \refN{E:iota} as follows.
Let $(f_0,\dots,f_{\ell -1})$ be an $\ell$-tuple of long knots.
Then we define
\[g=\iota(f_0,\dots,f_{\ell-1})\]
as the long link characterized by 
\[g|\BR\times{\{i\}}\,=\,\tau_{i}\circ H_{\min(\width(f_i),1/4)}\circ f_i,
\]
for $i=0,\cdots,\ell-1$, where $\tau_s$  is the vertical translation \refN{E:tau}.

The composition
\begin{equation}
\label{E:EmbKLK}
\Embd(\BR,\BR^N)^{\times\ell}
\stackrel{\iota}\longrightarrow
\Embd(\BR\times[\ell],\BR^N)
\stackrel{\rho}\longrightarrow
\Embd(\BR,\BR^N)^{\times\ell}
\end{equation}
is homotopic to the identity map, the homotopy being built using
homotopies
\[H_{\lambda\,+\,(1-\lambda)\cdot\min(\width(f_i),1/4)}\]
for $0\leq\lambda\leq1$.
Thus $\Embd(\BR,\BR^N)^{\times\ell}$ can be seen as a retract up to homotopy
of $\Embd(\BR\times[\ell],\BR^N)$.
By abuse of notation we look at $\iota$ as an inclusion and we
consider the \emph{pair of spaces}
\[\left(\Embd(\BR\times[\ell],\BR^N)\,,\,\Embd(\BR,\BR^N)^{\times\ell}\right).\]

\subsection{Long embeddings modulo immersions}
As often when dealing with spaces of embeddings, we will consider
these spaces \emph{modulo immersions}. More precisely, a \emph{long immersed
$1$-string} is a smooth immersion
\[f\colon\BR\looparrowright\BR^N\]
with boundary condtition \refN{E:BCK} and the space of long immersed
$1$-strings is denoted by
\[\Immd(\BR,\BR^N).\]
We have an inclusion
\[\Embd(\BR,\BR^N)\hookto \Immd(\BR,\BR^N)\]
and its homotopy fibre is denoted by
\[\KK:=\hofibre(\Embd(\BR,\BR^N)\hookto \Immd(\BR,\BR^N))\]
and called the \emph{space of  long knots modulo immersions}, or, by
abuse of terminology in this paper, simply the \emph{space of  long
  knots}.
Similarly we can consider long immersed $\ell$-strings
$g\colon\BR\times[\ell]\looparrowright\BR^N$
 and the homotopy fibre
\[\LL_\ell:=\hofibre(\Embd(\BR\times[\ell],\BR^N)\hookto \Immd(\BR\times[\ell],\BR^N))\]
which will be called in this paper the \emph{space of  long
  links with $\ell$ strings}.

The inclusion $\iota$ and retraction $\rho$ can readily be extended to
spaces of immersions on which thet are homotopy equivalences.
Thus we have a commutative diagram
\[
\xymatrix{
\Embd(\BR,\BR^N)^{\times\ell}                  \ar@{^{(}->}[r]^\iota
\ar@{^{(}->}[d]
&
\Embd(\BR\times{[\ell]},\BR^N)     \ar[r]^\rho
\ar@{^{(}->}[d]
&
\Embd(\BR,\BR^N)^{\times\ell}            
\ar@{^{(}->}[d]
\\
\Immd(\BR,\BR^N)^{\times\ell}                  \ar@{^{(}->}[r]^\iota_\simeq
&
\Immd(\BR\times{[\ell]},\BR^N)     \ar[r]^\rho_\simeq
&
\Immd\BR,\BR^N)^{\times\ell}.         
}
\]
Taking the homotopy fibres of the vertical inclusions
we get the sequence of spaces
\begin{equation}
\label{E:KLK}
\KK^{\times\ell}\stackrel{\iota}\longrightarrow
\LL_\ell\stackrel{\rho}\longrightarrow
\KK^{\times\ell}
\end{equation}
whose composite is homotopic to the identity.
We can look at the map $\iota$ as an inclusion 
and consider the pair of spaces
\[\left(\LL_\ell,\KK^{\times\ell}\right).\]

Notice also that by the classical Smale argument \cite{Sma:imm},
 we have homotopy
equivalences
\[
\Immd(\BR,\BR^N)^{\times\ell}\simeq 
\Immd(\BR\times{[\ell]},\BR^N)\simeq
(\Omega S^{N-1})^{\times\ell}
\]
Since the loop space $\Omega S^{N-1}$ is very well understood, at
least rationnaly,
the study of the sequence \refN{E:KLK} gives most of the information
about
the pair 
\[\left(\Embd(\BR\times[\ell],\BR^N)\,,\,\Embd(\BR,\BR^N)^{\times\ell}\right).\]
\section{Cosimplicial models for spaces of knots and links}
We recall the cosimplicial models for $\KK$ and $\LL_\ell$ given by
Sinha \cite{Sin:TSK} and Munson-Voli\'c \cite{MuVo:HoSt}.
These
cosimplicial spaces are built out of some compactifications of
configuration spaces in $\BR^N$ that we now review. 

First define the configuration space
\[\Confd(p,\BR^N):=\{(x_1,\dots,x_p)\in]-1,1[\times\BR^{N-1}: x_i\not=x_j\textrm{ for
}i\not=j\}\]
which is a subspace of $(]-1,1[\times\BR^{N-1})^p$.
Define, for integers $1\leq i<j\leq p$, the maps
\[\theta_{ij}\colon\Confd(p,\BR^N)\to
S^{N-1}\,,\,x\mapsto\frac{x_j-x_i}{\|x_j-x_i\|}\]
giving the direction between the $i$th and $j$th
components of the configuration $x$.

Following \cite{Sin:TSK}, define
\[\Confd\langle p,\BR^n\rangle\]
as the closure of the image of the injection
\[
\Confd(p,\BR^N)\stackrel{(incl,(\theta_{ij}))}\hookto
([-1,1]\times\BR^{N-1})^p\times\prod_{1\leq i< j\leq p}S^{N-1}.
\]
Notice that if we project the subspace
\[\Confd\langle p,\BR^n\rangle
\subset
([-1,1]\times\BR^{N-1})^p\times(S^{N-1})^{{p \choose 2}}
\]
to the factor $(S^{N-1})^{{p \choose 2}}$
we get exactly the $p$th term  of the Kontsevich
operad defined
by Sinha in \cite{Sin:OKS}.

These spaces form a cosimplicial space 
$\Confd\langle \bullet,\BR^n\rangle$
with cofaces given as maps
\[d_i\colon \Confd\langle p,\BR^n\rangle \to \Confd\langle
p+1,\BR^n\rangle\]
corresponding to ``doubling'' some points of the configuration, when
$1\leq i\leq p$; the first and last cofaces, $d_0$ and $d_{p+1}$,
are defined by inserting (or doubling) a point at
$(\mp1,0,\dots,0)\in[-1,1]\times\BR^{N-1}$.
This cosimplicial space models the space of long knots (see \refT{models} below.)

Similarly Munson-Voli\'c in \cite{MuVo:HoSt}
define a 
cosimplicial space 
$\Confd\langle \ell\bullet,\BR^n\rangle$
which models the space of links with $\ell$ strings.
Its $p$th term is
\[\Confd\langle \ell p,\BR^n\rangle\]
consisting of ``virtual'' configurations
$\left(x_{k,r})_{1\leq k\leq p\,,\,0\leq r\leq\ell-1)}\right)$.
The cofaces, for $1\leq i\leq p$
\[d_i\colon \Confd\langle \ell p,\BR^n\rangle \to \Confd\langle
\ell(p+1),\BR^n\rangle\]
consisting in doubling the $\ell$ components $x_{i,r}$ (for $r=0,\dots,\ell-1$) of the configuration.

The maps $\iota$ and $\rho$ from \refN{E:iota} and \refN{E:rho}
can readily be generalized to configuration spaces and we get maps of
cosimplicial spaces
\[
(\Confd\langle \bullet,\BR^n\rangle)^{\times\ell}
\stackrel{\iota}\longrightarrow
(\Confd\langle \ell \bullet,\BR^n\rangle)
\stackrel{\rho}\longrightarrow
(\Confd\langle \bullet,\BR^n\rangle)^{\times\ell}
\]
whose composite is homotopic to the identity map.

The 
\emph{homotopical totalization} of a cosimplicial space $X^\bullet$,
\[\hoTot(X^\bullet)m\]
is the mapping space
\[\map_\Delta(\widetilde{\Delta^\bullet},X^\bullet)\]
 where
$\widetilde{\Delta^\bullet}$ is a variation of the standard
cosimplicial space $\Delta^\bullet$ (more precisely
$\widetilde{\Delta^\bullet}$  is a cofibrant replacement of the
standard cosimplical space $\Delta^\bullet$ in the projective Quillen
model structure, see   \cite[Section 15]{McSm:cos}.)
Equivalently as the totalization of a functorial fibrant replacement
$\widehat{X^\bullet}$ of $X^\bullet$,
\[\hoTot(X^\bullet)=\Tot(\widehat{X^\bullet})=\map_\Delta(\Delta^\bullet,\widehat{X^\bullet}).\]
We define similarly the homotopy totalization of a map of cosimplicial spaces.

\begin{thm}[Sinha, Munson-Voli\'c]
\label{T:models}
Assume that $N\geq4$. Applying $\hoTot$ to the sequence
\[(\Confd\langle \bullet,\BR^n\rangle)^{\times\ell}
\stackrel{\iota}\longrightarrow
(\Confd\langle l\cdot\bullet,\BR^n\rangle)
\stackrel{\rho}\longrightarrow
(\Confd\langle \bullet,\BR^n\rangle)^{\times\ell}
\]
gives a sequence of spaces  weakly equivalent to 
\[\KK^{\times\ell}\stackrel{\iota}\longrightarrow
\LL_\ell\stackrel{\rho}\longrightarrow
\KK^{\times\ell}.\]
\end{thm}
\begin{proof}
The fact that the homotopy totalizations of 
$(\Confd\langle \bullet,\BR^n\rangle)^{\times\ell}$
and
 $\Confd\langle \ell\cdot\bullet,\BR^n\rangle$
are $\KK^{\times\ell}$ and $\LL_\ell$ is the content of Sinha and Munson-Voli\'c
papers. Sinha's proof for the model of long knots can be summarized as
follows. First restrict the
cosimplicial to its $p$th skeletton and note that the homotopy
totalisation of that skeletton is equivalent to the homotopy limit of a punctured
$p$-cubical diagram of configuration spaces.
 This cubical diagram is connected by a zigzag of weak 
equivalences to a cubical diagram whose spaces are essentially partial embedding spaces
\[\Embd([-1,1]\setminus\cup_j I_{i_j},[-1,1]\times\BR^{N-1})\]
where the $I_i$ are $p$ disjoint intervals inside $]-1,1[$.
Goodwillie-Klein multiple disjunction theory implies then that this
homotopy limit is equivalent (up to some degree going to infinity with
$p$)
to the space of embeddings
\[\Embd([-1,1],[-1,1]\times\BR^{N-1})=\Embd(\BR,\BR^{N}).\]
The proofs for long links is completely analogous.

It is easy to check that at each level of the above zigzag equivalences we
can generalize the maps $\iota$ and $\rho$.  This proves the result
\end{proof}

\section{The cohomology spectral sequence of a cosimplicial space and
  Poincar\'e and Euler series}

Let $X^\bullet$ be a cosimplicial space.  There is an associated 
first-quadrant spectral sequence converging (under some hypothesis as
in \refP{convHBKSS} below) to the cohomology of its homotopical totalization
\[\{E^{p,q}_r\}_{r\geq1}\Longrightarrow \Ho^{q-p}(\hoTot(X^\bullet))\]
with coefficients in a fixed field $\BK$.
The $E_1$ page of this spectral sequence is given by
\begin{equation}
\label{E:defE1}
E_1^{p,q}:=\frac{\Ho^q(X^p)}{\sum_{i=1}^p\im(\Ho^*(s_i))}
\end{equation}
where 
\[s_i\colon X^p\to X^{p-1}\]
are the codegeneracies of the cosimplicial space.
This is the \emph{cohomology  Bousfield-Kan spectral sequence},  
or \HBKSS for short, see \cite{Bou:HSS}. Sometimes, when $T$ is the
homotopy totalization of the cosimplicial space $X^\bullet$, we will 
denote this spectral sequence by
\[E_r^{p,q}\{T\}\]
to emphasize the underlying cosimplicial space (assuming that the $T$
makes clear the cosimplicial space we are talking about.)

We recall now a convergence condition for the \HBKSS.
We say that a positive real number  $\alpha>0$ is a \emph{lower slope}
for a first quadrant spectral sequence $\{E_r^{*,*}\}_{r\geq1}$
if 
\[\forall p,q\in\BN:\quad q<\alpha p\quad\Longrightarrow\quad E_1^{p,q}=0.\]
This means that the $E_1$ page is concentrated above a line of
that slope.

\begin{prop}[Bousfield\cite{Bou:HSS}]
\label{P:convHBKSS}
If the \HBKSS has a lower slope $\alpha>1$ then it converges
to $\Ho^{*}(\hoTot(X^\bullet))$.
\end{prop}
We also say that a real number  $\beta>0$ is an \emph{upper slope}
for a first quadrant spectral sequence $\{E_r^{*,*}\}_{r\geq1}$
if 
\[\forall p,q\in\BN:\quad q>\beta p\quad\Longrightarrow\quad E_1^{p,q}=0.\]


The following result is folklore and easy to prove from the definition
\refN{E:defE1}.
\begin{prop}
\label{P:dimE1}
\[\dim E^{p,q}_1=\sum_{r=0}^p(-1)^{r-p} {p\choose r}\dim \Ho^q(X^r),\]
where ${p \choose r}$ is the binomial coefficient.
\end{prop}

We introduce the following notation for Poincar\'e series.
Let $V^*=\oplus_{k=0}^\infty V^k$ be a graded vector space. Then its
Poincar\'e series is denoted by 
\begin{equation}
\label{E:PS}
V^{\uu{*}}[x]:=\sum_{k=0}^\infty \dim(V^k)x^k.
\end{equation}
Here the underlined asterisque $\uu{*}$ is a reminder that the
Poincar\'e series is taken along this gradation, and the argument in
brackets
$[x]$ is the variable of the Poincar\'e series.
For example, the usual Poincar\'e series of a space $X$  is
\[\Ho^{\uu{*}}(X)[x]=\sum_{q=0}^\infty\dim \Ho^q(X)x^q.\]
For a bigraded vector space, as $E_1^{*,*}$ we will consider the double 
Poincar\'e series
\[E_1^{\uu{*},\uuu{*}}[u,x]:=\sum_{p=0}^\infty\sum_{q=0}^\infty \dim(E_1^{p,q})u^px^q,\]
the simply underlined asterisk is the degree ($p$) corresponding to the
\emph{first} variable ($u$), and the doubly underlined asterique is the
degree ($q$)
 corresponding to the \emph{second} variable ($x$).

Using notation \refN{E:PS}, \refP{dimE1} implies the following
\begin{cor}
\label{C:HBKSSseries}
For the \HBKSS of a cosimplicial space $X^\bullet$,
\[ E_1^{p,\uu{*}}[x]=\sum_{i=0}^p(-1)^{i-p} {p\choose i}
\Ho^{\uu{*}}(X^i)[x].\]
\end{cor}

We introduce now a notation for the Euler characteristics and the
``Euler series''.
For a graded vector space of finite total dimension, $V^*=\oplus_{k=0}^\infty
V^k$, its Euler characteristic is denoted by
\[\chi(V^\bullet):=\sum_{k=0}^\infty(-1)^k\dim(V^k),\]
hence the bullet is the degree along which is taken the Euler
characteristic. Thus for a bigraded object $E^{*,*}$ we set
\[\chi(E^{\bullet,q}):=\sum_{p=0}^\infty(-1)^p\dim(E^{p,q}).\]
We define the \emph{Euler series} of a bigraded object (of finite
type) as
\[
\chi(E^{\bullet,\uu{*}})[x]:=\sum_{q=0}^\infty\left(\sum_{p=0}^\infty(-1)^p\dim(E^{p,q})\right)\,x^q.
\] 
In other words we have
\[
\chi(E^{\bullet,\uu{*}})[x]\,=\,
E^{\uu{*},\uuu{*}}[-1,x].
\]

From \refP{dimE1} we deduce
\begin{prop}
\label{P:chiSE1}
Assume that the \HBKSS of $X^\bullet$ has a lower slope $\alpha>1$ and
an upper slope $\beta<\infty$ and each space $X^i$ is of finite type.
Then
\[\chi(E^{\bullet,\uu{*}}_1)[x]=\sum_{p,i\geq0}(-1)^i{p\choose i}\left(\Ho^{\uu{*}}(X^i)[x]\right).\]
\end{prop}
\begin{proof}
First we show that the double series on the right is well defined. 
This is a consequence of the fact that ${p \choose i}=0$ for $i>p$,
the bounds on the slopes, and the finite type hypothesis.
The equality is then a direct consequence of  \refC{HBKSSseries}.
\end{proof}

One advantage  of the Euler series of $E_1$ is that it can gives lower
bounds on the Betti numbers of the homotopy totalization when the
spectral sequence collapses at the $E_2$-term, as we explain now.
 
Assume that the \HBKSS has a lower slope $\alpha>1$ and is of finite
type, that is
$\dim E_1^{p,q}<\infty$ for all $p,q$.
The totalization of the $r$th page of the spectral sequence is
the graded vector space $(\Tot E^{*,*}_r)$ defined by
\[\Tot(E_r^{*,*})^n=\oplus_{p\geq0}E_r^{p,p+n}\,=\,\oplus_{p=0}^{n/(\alpha-1)}E_r^{p,p+n},\]
which is finite dimensional.
When $r=\infty$ or when the spectral sequence collapses at page $E_r$, for
some $r\in\BN$,
then, by \refP{convHBKSS},
\[\Ho^n(\hoTot X^\bullet)\cong \Tot(E_r^{*,*})^n.\]

Since the differential on the page $E_1$ is horizontal, that is of bidegree
$(-1,0)$, the Euler characteristcs of the lines of the spectral
sequence on pages $E_1$ and $E_2$ are the same:
\[\chi(E_1^{\bullet,q})=\chi(E_2^{\bullet,q}).\]

We then have
\begin{prop}
\label{P:Tot>chi}
Let $E^{*,*}_1$ be a spectral sequence of finite type and with a lower
slope $\alpha>1$.
Then
\[
\sum_{k=\lceil n(1-1/\alpha)\rceil}^n\dim(\Tot(E_2^{*,*})^k
\geq
\left|\chi(E_1^{\bullet,n})\right|.\]
where we denote by $\lceil x\rceil$ the smallest integer $\geq x$.
\end{prop}

The previous lemma implies that lower bounds for the Betti numbers of
the homotopy totalization of a cosimplicial space $X^\bullet$ can be
obtained,
when the \HBKSS collapses at the $E_2$-page, by computing
Euler characteristics at the level of the $E_1$-page.

\section{Euler Poincar\'e series in the \HBKSS for spaces of links.}
Here is our central result.
\begin{thm}
\label{T:chiE2L}
Let $N\geq4$ and consider the \HBKSS of the Munson-Voli\'c cosimplicial
model for the space of long links with $\ell$ strings..
Then  its first page $E_1^{*,*}\{\LL_\ell\}$  has a lower slope $>1$, a finite upper slope and the Euler
series
\[\chi(E^{\bullet,\uu{*}}_1 \{\LL_\ell\})[x]=\frac{1}{(1-x^{N-1})(1-2x^{N-1})\cdots(1-\ell x^{N-1})}.
\]
\end{thm}

In order to prove this theorem, consider first the Poincar\'e series
of configuration spaces in $\BR^N$.
Fadelland Neurwith in \cite{FaNe:con} have computed that
\begin{equation}
\label{E:PSConf}
H^{\uu{*}}(\Conf(k,\BR^N)[x]=(1-x^{N-1})(1-2x^{N-1})\dots(1-(k-1)x^{N-1}).
\end{equation}

This Poincar\'e series is closely related to Stirling numbers that we quickly
review now. They are integer numbers
\[\stirlingone nk\quad\textrm{ and}\quad\stirlingtwo nk,\]
named\emph{ Stirling numbers of the first and of the second kind},
and defined for example in \cite[Section 6.1]{GKP:CM}.
The
convention about the signs of these numbers varies in the litterature,
 and we choose the
convention of \cite{GKP:CM} where all these numbers are non-negative.
These numbers are defined not only for non-negative integers arguments,
$k,n\in\BN$,
but are extented to all integers by 
\[\stirlingone kn=\stirlingtwo kn=0\textrm{ when $k$ and $n$ have
  opposite signs}
\]
and
\[\stirlingone kn=\stirlingtwo {-n}{-k},\]
(see \cite[(6.33)]{GKP:CM})

We have then
\begin{equation}
\label{E:genS1}
\sum_{j=0}^k\stirlingone{k}{k-j}u^j=(1+u)(1+2u)\dots(1+(k-1)u),
\end{equation}
and replacing $u$ by $x^{N-1}$ in the right hand side we get the
Poincar\'e series of $\Conf(k,\BR^N)$, see \refN{E:PSConf}.

We also have the following generating series for the Stirling numbers
of the second kind,
\begin{equation}
\label{E:genS2}
\sum_{j=0}^\infty\stirlingtwo{k+j}{k}\,u^j
=
\frac1{(1-u)(1-2u)\cdots(1-ku)}.
\end{equation}

We establish the following relation between Stirling numbers of the first and
second kind, which will be the key for our main result. This
combinatotail result is
maybe new. The below proof is due to Manuel Kauers.
\begin{prop}
\label{P:S1S2}
For $j,\ell\geq1$,
\[\sum_{p=0}^\infty\sum_{r=0}^p(-1)^r{p\choose
  r}\stirlingone{\ell r}{\ell r-j}
=
\stirlingtwo{\ell+j}\ell.\]
\end{prop}
The left hand side make sense because almost all terms in the exterior
infinite sum are $0$, as we will see in the proof.

In order to prove this proposition, we need the following lemma.

\begin{lemma}[M.~Kauers]
\label{L:Kauers}
Let $q$ be a polynomial. Then
\[\sum_{p=0}^\infty\left(\sum_{r=0}^p(-1)^r{p\choose r}q(r)\right)=q(-1)\]
where all but finitely many terms of the exterior infinite sum in the left hand
side are $0$.
\end{lemma}
\begin{proof}[Proof of \refL{Kauers}]
Consider the generating series
\begin{equation}
\label{E:s}
s(x):=\sum_{p=0}^\infty\left(\sum_{r=0}^p(-1)^r{p\choose
    r}q(r)\right)x^p.
\end{equation}
We need to prove that $s(x)$ is a finite sum and that $s(1)=q(-1)$.
We first will prove it for the special poynomials $q=q_{(d)}$, for
$d\in\BN$, with
\begin{equation}
\label{E:qd}
q_{(d)}(r):=(r+1)(r+2)\cdots(r+d),
\end{equation}
that is $q_{(0)}(r)=1$, $q_{(1)}(r)=r+1$, $q_{(2)}(r)=r^2+3r+2$, ...
Then we will deduce the general case by linearity.

Since $q_{(0)}(r)=1$, we have
\[\sum_{r=0}^\infty q_{(0)}(r)\,y^r\,=\,\frac{1}{1-y}.
\]
Taking the $d$th derivative on both sides we get
\begin{equation}
\label{E:ad}
\sum_{r=0}^\infty q_{(d)}(r)\,y^r\,=\,\frac{d!}{(1-y)^{d+1}}.
\end{equation}

The series
\[a(y):=\sum_{r=0}^\infty q(r)\,y^r\]
is the Euler transform (or binomial transform) of the generating
series $s(x)$ of \refN{E:s}.
Therefore
\begin{equation}
\label{E:bintr}
s(x)=\frac1{1-x}\,a\left(\frac{x}{x-1}\right).
\end{equation}
When $q=q_{(d)}$, we deduce from \refN{E:ad} and \refN{E:bintr} that
\begin{eqnarray*}
s(x)&=&\frac{1}{1-x}\cdot\frac{d!}{\left(1-\frac{x}{x-1}\right)^{d+1}}\\
&=&d!(1-x)^d.
\end{eqnarray*}
Thus, for $q=q_{(d)}$, $s(x)$ is a polynomial, hence a finite sum.
Moreover 
\[
s(1)=
\begin{cases}
1&\textrm{ if }d=0\\
0&\textrm{ if }d\geq1,
\end{cases}
\]
and we compute from \refN{E:qd}
that
\[
q(-1)=
\begin{cases}
1&\textrm{ if }d=0\\
0&\textrm{ if }d\geq1.
\end{cases}
\]
Therefore $s(1)=q(-1)$ when $q=q_{(d)}$, which implies the lemma for
those polynomials. 

Any polynomial $q$ is a linear combination of the polynomials
$q_{(0)}$, $q_{(1)}$, $q_{(2)}$, ...,
and both sides of the equation in the statement of the lemma are
linear
in the polynomials $q$. Therefore the lemma is true for all
polynomials $q$.
\end{proof}

\begin{proof}[Proof of \refP{S1S2} (M.~Kauers)]
Fix $j,\ell\geq1$. Notice that
\[q(r):=\stirlingone{\ell r}{\ell r-j}\]
is a polynomial in $\ell r$, hence in $r$, because of the following equation from
\cite[(6.44) in Section 6.2]{GKP:CM}:
\[\stirlingone{x}{x-n}=\sum_{i=0}^{n}\eulernb ni{x+i\choose 2n}\]
where $\eulernb ni$ are the Euler numbers of the second order, and
\[
{x+i\choose 2n}=\frac{(x+i)(x+i-1)\cdots(x+i-2n+1)}{(2n)!}
\]
is a polynomial in $x$ of degree $2n$.

By \refL{Kauers}, we deduce that the left hand side of the equation in
the statement of the proposition is
\[q(-1)=\stirlingone{-\ell}{-\ell-j},\]
which is the same as $\stirlingtwo{\ell+j}{\ell}$ by \cite[(6.33) in Section
6.1]{GKP:CM}.
\end{proof}

\begin{proof}[Proof of \refT{chiE2L}]
Consider the Munson-Voli\'c cosimplicial space
$\Confd\langle\ell\bullet,\BR^N\rangle$.
By \refN{E:PSConf} and \refN{E:genS1}, the Poincar\'e series of the $p$th term of the
cosimplicial space is
\begin{eqnarray}
\notag
\Ho^{\uu*}(\Confd\langle\ell p,\BR^N\rangle)[x]
&=
&
\prod_{j=1}^{\ell p-1}(1+j\,x^{N-1})\\
\notag
&=
&
\sum_{j=0}^{lp}
\stirlingone{\ell p}{\ell p-j}(x^{N-1})^j\\
\label{E:PSConflp}
&=
&
\sum_{j\geq0}
\stirlingone{\ell p}{\ell p-j}(x^{N-1})^j.
\end{eqnarray}
Using \refP{chiSE1}, \refN{E:PSConflp}, \refP{S1S2}, and \refN{E:genS2},
we compute that the Euler series of the first page of the \HBKSS is 
\begin{eqnarray*}
\chi(E^{\bullet,*}_1)[x]
&=&
\sum_{p=0}^\infty
\sum_{r=0}^p
(-1)^r{p\choose r}
\Ho^{\uu*}(\Confd\langle\ell r,\BR^N\rangle)[x]
\\
&=&
\sum_{p=0}^\infty
\sum_{r=0}^p
(-1)^r{p\choose r}
\sum_{j\geq0}
\stirlingone{\ell r}{\ell r-j}(x^{N-1})^j
\\
&=&
\sum_{j\geq0}
\left(\sum_{p=0}^\infty
\sum_{r=0}^p
(-1)^r{p\choose r}\stirlingone{\ell r}{\ell r-j}
\right)
(x^{N-1})^j
\\
&=&
\sum_{j\geq0}
\stirlingtwo{\ell +j}{j}
(x^{N-1})^j
\\
&=&
\frac{1}{(1-x^{N-1})(1-2x^{N-1})\dots(1-\ell x^{N-1})}.
\end{eqnarray*}
\end{proof}

\begin{cor}
\label{C:chiE2Kl}
Let $N\geq4$ and consider the \HBKSS of the cosimplicial $\ell$-fold
product of Sinha's
model for the space of long knots.
Then its first page $E_1^{*,*}\{\KK^\ell\}$  has a lower slope $>1$, a finite upper slope and the Euler
series
\[\chi(E^{\bullet,\uu{*}}_1 \{\KK^\ell\})[x]=\frac{1}{(1-x^{N-1})^\ell}\]
\end{cor}
\begin{proof}
Given a cosimplicial space $X^\bullet$, one can consider its
$\ell$-fold product $(X^\bullet)^{\times\ell}$ which is the
cosimplicial space defined with its $p$th term given by
\[(X^p)^{\times\ell}=X^p\times\dots\times X^p.\]
and the obvious cofaces and codegeneracies induced by those on
$X^\bullet$.

By the Kunneth formula, the Euler  series of the $E_1$-page of the \HBKSS of
$(X^\bullet)^{\times\ell}$
is then the $\ell$-th power of the Euler  series for $X^\bullet$.
Since the cosimplicial model for long knots is exactly the
cosimplicial models for long links with $1$ string and since
the Euler series for those have been computed in
\refT{chiE2L} to be
\[\frac{1}{(1-x^{N-1})},\]
the corollary follows.
\end{proof}

We finish this paper by some application to the growth of the
dimension of the totalisation of the $E_2$ page of a spectral sequence
computing
\[\Ho^*(\LL_\ell,\KK^{\times\ell}).\]
By \refN{E:KLK}, we have a pair of spaces $(\LL_\ell,\KK^{\times\ell})$
in which the subspace is a retract (up to homotopy) of the bigger
space $\LL_\ell$. 
Since the retraction is at the level of the cosimplicial models, we
have a spectral sequence $E_r^{*,*}\{(\LL_\ell,\KK^{\times\ell})\}$
 which converges  to $\Ho^*(\LL_\ell,\KK^{\times\ell})$, and
the spectral sequence is a quotient
\[E_r^{*,*}\{(\LL_\ell,\KK^{\times\ell})\}\cong
E_r^{*,*}\{\LL_\ell\}\,/E_r^{*,*}\{\KK^{`*\ell}\}.
\]
Therefore we have
\[
\chi\left(E_2^{\bullet,\uu*}\{(\LL_\ell,\KK^{\times\ell})\}\right)[x]\,=\,
\chi\left(E_2^{\bullet,\uu*}\{\LL_\ell\}\right)[x]\,-\,
\chi\left(E_2^{\bullet,\uu*}\{\KK^{\times\ell}\}\right)[x].
\]

By \refT{chiE2L} and \refC{chiE2Kl} we deduce that
\[\chi\left(E_2^{\bullet,\uu*}\{(\LL_\ell,\KK^{\times\ell})\}\right)[x]\,=\,
\frac{1}{(1-x^{N-1})(1-2x^{N-1})\cdots(1-\ell x^{N-1})}\,-\,
\frac{1}{(1-x^{N-1})^\ell}.\]
The coefficients of that series have an exponential growth of rate
$(\ell)^{1/(N-1)}>1$.  By \refP{Tot>chi} the Betti numbers of the
totalization of the $E_2$ page of the spectral sequence have the same
growth.

We conjecture that this spectral sequence collapses at the $E_2$-page
(this is known to be the case for the spectral sequence for long
knots, and probably the same method will work for the space of long
links, hence also for the spectral sequence of the pair 
$(\LL_\ell,\KK^{\times\ell})$.)

\bibliographystyle{plain}
\def\cprime{$'$}






\end{document}